\documentclass[12pt]{amsart}

\makeatletter
\@namedef{subjclassname@2020}{%
  \textup{2020} Mathematics Subject Classification}
\makeatother

\usepackage[utf8]{inputenc}
\usepackage{amsmath}
\usepackage{amsfonts} 
\usepackage{amsthm}
\usepackage{verbatim}
\usepackage{etoolbox}
\usepackage{dirtytalk}
\usepackage{tikz-cd}
\usepackage{derivative}
\usepackage{amssymb}

\theoremstyle{plain}
\newtheorem{thm}{Theorem}[section]
\newtheorem{lem}[thm]{Lemma}
\newtheorem{obs}[thm]{Observation}
\newtheorem{prop}[thm]{Proposition}

\newtheorem*{claim*}{Claim}

\newtheorem{innerthmABC}{Theorem}
\newenvironment{thmABC}[1]
  {\renewcommand\theinnerthmABC{#1}\innerthmABC}
  {\endinnerthmABC}

\theoremstyle{definition}

\newtheorem{ex}[thm]{Example}

\theoremstyle{remark}
\newtheorem{rem}{Remark}[thm]

\AtBeginEnvironment{thm}{\setcounter{equation}{0}}
\AtBeginEnvironment{lem}{\setcounter{equation}{0}}
\AtBeginEnvironment{obs}{\setcounter{equation}{0}}
\AtBeginEnvironment{defi}{\setcounter{equation}{0}}
\AtBeginEnvironment{prop}{\setcounter{equation}{0}}
\AtBeginEnvironment{cor}{\setcounter{equation}{0}}

\AtBeginEnvironment{thm}{\setcounter{case}{0}}
\AtBeginEnvironment{lem}{\setcounter{case}{0}}
\AtBeginEnvironment{obs}{\setcounter{case}{0}}
\AtBeginEnvironment{defi}{\setcounter{case}{0}}
\AtBeginEnvironment{prop}{\setcounter{case}{0}}
\AtBeginEnvironment{cor}{\setcounter{case}{0}}

\AtBeginEnvironment{thm}{\setcounter{claim}{0}}
\AtBeginEnvironment{lem}{\setcounter{claim}{0}}
\AtBeginEnvironment{obs}{\setcounter{claim}{0}}
\AtBeginEnvironment{defi}{\setcounter{claim}{0}}
\AtBeginEnvironment{prop}{\setcounter{claim}{0}}
\AtBeginEnvironment{cor}{\setcounter{claim}{0}}

\newcommand{\R}{\mathbb{R}}
\newcommand{\K}{\mathbb{K}}
\newcommand{\RR}{\mathcal{R}}

\begin{document}
\baselineskip=17pt

%%%%%%%%%%%%%%%%

\title{Separately regular and separately Nash functions}

\author[J. Banecki]{Juliusz Banecki}
\address{Faculty of Mathematics and Computer Science\\
Jagiellonian University\\
ul. Lojasiewicza 6\\
30-348 Krakow, Poland}
\email{juliusz.banecki@student.uj.edu.pl}

\date{}

\begin{abstract}
We strengthen certain known results saying that separately regular functions are rational and separately Nash functions are semialgebraic. The approach presented here unifies and highlights the similarities between the two problems.
\end{abstract}


\maketitle

\section{Introduction}
Given three sets $X,Y,Z$, a function $f:X\times Y\rightarrow Z$ and points $x_0\in X,y_0\in Y$ we denote the functions $x\mapsto f(x,y_0)$ and $y\mapsto f(x_0,y)$ by $f(\cdot,y_0)$ and $f(x_0,\cdot)$ respectively. By $\Gamma_f$ we denote the graph of $f$. 

Let $\K$ be an infinite field. By an affine variety over $\K$ we mean a Zariski closed subset $X$ of $\K^n$ for some $n$. A function $f:X\rightarrow \K$ is polynomial if it is the restriction of some polynomial function on $\K^n$. The algebra of polynomial functions on $X$ is denoted by $\mathcal{P}(X)$. The function $f$ is said to be regular if for every point $x\in X$ there exist two polynomial functions $P,Q\in\mathcal{P}(X)$ such that $Q(x)\neq 0$ and $Qf=P$ holds on the entire variety $X$. The algebra of regular functions on $X$ is denoted by $\RR(X)$. Note that by Hilbert's Nullstellensatz $\mathcal{P}(X)=\RR(X)$ if the field $\K$ is algebraically closed.

Following \cite[Definition 7.2]{Palais} an irreducible algebraic affine variety $X$ is said to be \emph{algebraically of the second category}, if whenever it is covered by a countable family of subsets
\begin{equation*}
    X=\bigcup_{i=1}^\infty X_i
\end{equation*}
at least one of the sets $X_i$ is Zariski dense in $X$.
\begin{ex}\label{alg_second_category}
According to \cite[Theorem 9.3]{Palais} if $\K$ is an uncountable algebraically closed field then every irreducible affine algebraic variety $X$ over $\K$ is algebraically of the second category. Also thanks to \cite[Lemma 10]{Gwozdziewicz} if $\K$ is any uncountable field then $X=\K^n$ is algebraically of the second category for every $n\geq 0$.
\end{ex}
\begin{rem}
It could be tempting to think that every irreducible variety over an uncountable field is algebraically of the second category. This turns out to be false as shown by the following example:

Choose $E$ to be the affine part of any non-singular elliptic curve in the Weierstrass form defined over $\mathbb{Q}$, for example we can take one given by the equation $P(x,y)=0$ where $P(x,y)=y^2-x^3-x$. Let $\K_0$ be any algebraic number field such that the curve has infinitely many $\K_0$-rational points (such a number field does exist as it suffices to extend $\mathbb{Q}$ by the coordinates of any non-torsion point of $E$). Define $\K$ as $\K_0$ extended by an uncountable number of algebraically independent elements:
\begin{equation*}
    \K:=\K_0(\{x_i\}_{i\in \R}).
\end{equation*}
Now consider $E_\K$, the curve given by $P(x,y)=0$ over $\K$. We claim that each point of $E_\K$ is $\K_0$-rational. Indeed, choose any point of $E_\K$ in the form $(Q_1,Q_2)$, where $Q_1$ and $Q_2$ are rational functions in the variables $\{x_i\}_{i\in \R}$. These functions depend only on finitely many of the variables, say $\{x_1,\dots,x_n\}$. As elliptic curves are known not to be unirational, the following mapping must be constant:
\begin{align*}
    \mathbb{C}^n&\dashrightarrow E \\
    (x_1,\dots,x_n)&\mapsto (Q_1(x_1,\dots,x_n),Q_2(x_1,\dots,x_n)).
\end{align*}
This shows that $Q_1,Q_2\in \K_0$, so the point is indeed $\K_0$-rational.

This in particular shows that $E_\K$ has countably infinitely many points. It may not be clear that $E_\K$ is irreducible, so choose $F$ to be some irreducible component of $E_\K$ consisting of infinitely many points. As it is countable, it is an irreducible variety over an uncountable field, which is not algebraically of the second category.
\end{rem}

The first main result of this paper is the following strengthening of a result contained in \cite{Gwozdziewicz}:
\begin{thmABC}{A}\label{thm_reg}
Let $\K$ be a field and let $X,Y$ be two affine irreducible varieties over $\K$. Suppose that $Y$ is algebraically of the second category. Let $A\subset X$ be a Zariski dense set. Suppose that a function $f:X\times Y\rightarrow \K$ satisfies the following conditions:
\begin{enumerate}
    \item $f(\cdot,y)$ is regular for every $y\in Y$,
    \item $f(x,\cdot)$ is regular for every $x\in A$.
\end{enumerate}
Then $f$ is rational, i.e. there exist two polynomials $P,Q\in\mathcal{P}(X\times Y)$ with $Q$ non-zero such that $Qf=P$ on $X\times Y$.
\end{thmABC}
The proof is based on ideas contained in \cite{Gwozdziewicz}. Related results for the field of real numbers and an arbitrary algebraically closed field are given in \cite[Theroem 5.1]{curve} and \cite[Theorem 8.3]{Palais}, respectively.

Let now $R$ be a real closed field. A subset $X\subset R^n$ for some $n$ is said to be semialgebraic if it is given by some sign conditions on a finite family of polynomials in $\mathcal{P}(R^n)$. Let $X\subset R^n$ be a semialgebraic set open in the Euclidean topology. A function $f:X\rightarrow R$ is said to be Nash if it is infinitely differentiable and semialgebraic. For an introductory treatment of the topic see \cite[Section 8]{RAG}. 

We prove the following version of Theorem \ref{thm_reg} for Nash functions:
\begin{thmABC}{B}\label{thm_Nash}
Let $R$ be an uncountable real closed field. Let $X$ and $Y$ be two semialgebraic open subsets of $R^{n_X}$ and $R^{n_Y}$ respectively, for some integers $n_X,n_Y\geq 0$. Assume that $X$ is semialgebraically connected. Let $A\subset X$ be a set dense in the Zariski topology on $R^{n_X}$. Suppose that a function $f:X\times Y\rightarrow R$ satisfies the following conditions:
\begin{enumerate}
    \item $f(\cdot,y)$ is Nash for every $y\in Y$,
    \item $f(x,\cdot)$ is Nash for every $x\in A$.
\end{enumerate}
Then $f$ is semialgebraic.
\end{thmABC}
This is a generalisation of some results contained in \cite{separatelyNash}. The proof is based on ideas coming from \cite{Gwozdziewicz}, however in this case much more work needs to be done than in the proof of Theorem \ref{thm_reg}.

\section{Proofs}
\subsection{The main lemma}
We need to state a lemma which will play a role both in the proof of Theorem \ref{thm_reg} and Theorem \ref{thm_Nash}. It seems that the lemma may be useful in proving other results similar to the ones in the paper, and for this reason it is stated in a rather general form.

We need to introduce some notation. Let $\K$ be a field and let $S$ be a set. We denote by $\K^S$ the vector space of all $\K$ valued functions on $S$. Let $V\subset \K^S$ be some vector subspace. We denote by $T^k V$ the $k$th tensor power of $V$, which naturally embeds in $\K^{S^k}$ with the embedding defined as the linear extension of
\begin{equation*}
    (Q_1\otimes \dots \otimes Q_k)(x_1,\dots,x_k):=Q_1(x_1) \cdots   Q_k(x_k).
\end{equation*}
Suppose further that $V$ admits a countable basis $(e_i)_{i\in\mathbb{N}}\in V$ (we take $\mathbb{N}=\{1,2,\dots\}$). Let $(e_i^\ast)_{i\in \mathbb{N}}\in V^\ast$ denote the duals. For an element $Q\in V\backslash\{0\}$ define
\begin{equation*}
    d(Q)=\max\{i: e_i^*(Q)\neq 0\}.
\end{equation*}
For a subset $F\subset S$ define 
\begin{equation*}
    c(F):=\min\{d(Q):Q\in V\backslash\{0\},\;Q\vert_F\equiv 0\}
\end{equation*}
(take $c(F)=\infty$ if the set is empty).

\begin{lem}\label{main_lemma}
Let $S$ be a set and $V\subset \K^S$ a vector space with countable basis $(e_i)_{i\in\mathbb{N}}$ as above. Let $n$ be a natural number. There exist $n$ elements $(P_i)_{i=1,\dots,n}$ of $T^{n-1}V$ such that for any subset $F\subset S$ satisfying $c(F)=n$ and any $n-1$ points $x_1,\dots,x_{n-1}\in F$ if we define
\begin{equation*}
    P:=\sum_{i=1}^n P_i(x_1,\dots,x_{n-1})e_i\in V
\end{equation*}
then $P\vert_F\equiv 0$. Moreover, for a given set $F$ the points $x_1,\dots,x_{n-1}$ can be chosen in such a way that $P\neq 0$.

\begin{proof}
Fix the set $F\subset S$. Consider the subspace of $\K^F$ spanned by the functions $e_1\vert_F,\dots,e_n\vert_F$. Its dimension is precisely $n-1$ as the only non-trivial relation between the vectors is given by the coefficients of an element $Q\in V\backslash\{0\}$ satisfying $d(Q)=n,Q\vert_F\equiv 0$.

It is a simple fact from linear algebra that we can find $n-1$ distinct points $x_1,\dots,x_{n-1}\in F$ such that
\begin{equation*}
    e_1\vert_{\{x_1,\dots,x_{n-1}\}},\dots,e_n\vert_{\{x_1,\dots,x_{n-1}\}}
\end{equation*}
span a subspace of $\K^{\{x_1,\dots,x_{n-1}\}}$ of dimension $n-1$. This precisely means that the matrix 
\begin{equation*}
    M:=\begin{bmatrix}
        e_1(x_1) & \dots & e_n(x_1) \\
        \vdots & \ddots & \vdots \\
        e_1(x_{n-1}) & \dots & e_n(x_{n-1})
    \end{bmatrix}
\end{equation*}
is of rank $n-1$. 

As already said, up to scalar multiplication the unique non-zero vector in the kernel of $M$ is given by $[e_i^\ast(Q)]_{i=1,\dots,n}$. On the other hand if by $\delta_i$ we denote the determinant of $M$ with $i$-th column removed times $(-1)^i$ then the vector $[\delta_i]_{i=1,\dots,n}$ also belongs to the kernel and is non-zero. Note that each $\delta_i$ can be written as
\begin{equation*}
    \delta_i=P_i(x_1,\dots,x_{n-1})
\end{equation*}
where $P_i$ is an element of $T^{n-1}V$ depending only on $n$ and $i$. Now if we consider
\begin{equation*}
    P:=\sum_{i=1}^n P_i(x_1,\dots,x_{n-1})e_i
\end{equation*}
then $P$ is equal to $Q$ scaled by some factor. This shows that it
satisfies $P\vert_F\equiv 0$. Note that this equation stays true for any choice of $x_1,\dots,x_{n-1}$ for which the matrix $M$ is of rank $n-1$, and for such a choice $P$ is a non-zero element of $V$. On the other hand, if $x_1,\dots,x_{n-1}$ are chosen in such a way that the rank of $M$ is smaller than $n-1$ then $P\equiv 0$ so the equation $P\vert_F\equiv 0$ stays true. 
\end{proof}
\end{lem}

\subsection{Proof of Theorem \ref{thm_reg}}

\begin{proof}[Proof of Theorem \ref{thm_reg}]
Let $S:=X\times \K$ and let $V$ be the vector space of all polynomials on $X\times \K$ which are of degree at most one with respect to the second variable (i.e. of functions of the form $P_1(x)+P_2(x)t$, where $P_1,P_2\in\mathcal{P}(X)$). Fix some basis $(e_i)_{i\in\mathbb{N}}$ of $V$. For $Q\in V\backslash\{0\}$ define $c(Q)$ as in Lemma \ref{main_lemma} and for a subset $F\subset X\times K$ define $d(F)$ also as in Lemma \ref{main_lemma}.

Let $f:X\times Y\rightarrow \K$ be a function satisfying the assumptions of Theorem \ref{thm_reg}. By assumption $Y$ is algebraically of second category, so there exists a natural number $n$ such that the set
\begin{equation*}
    B:=\{y\in Y:c(\Gamma_{f(\cdot,y)})=n\}
\end{equation*}
is Zariski dense in $Y$. Fix $y_0\in B$. Applying Lemma \ref{main_lemma} with $F=\Gamma_{f(\cdot,y_0)}$, we can find polynomials $(P_i)_{i=1,\dots,n}$ in $\mathcal{P}((X\times\K)^{n-1})$ such that for all points $x_1,\dots,x_{n-1}\in X$ the polynomial
\begin{equation*}
    P_{y_0}:=\sum_{i=1}^n P_i(x_1,f(x_1,y_0),\dots,x_{n-1},f(x_{n-1},y_0))e_i\in V
\end{equation*}
satisfies $P_{y_0}(x,f(x,y_0))=0$ for $x\in X$. Now the set of tuples $(x_1,\dots,x_{n-1})\in X^{n-1}$ for which $P_{y_0}$ is non-trivial is Zariski open in $X^{n-1}$ and by assumption is non-empty, so we can find such a tuple which belongs to $A^{n-1}$. Fix such a choice of $(x_1,\dots,x_{n-1})$. 

The equation $P_y(x,f(x,y))=0$ by construction holds for every $(x,y)\in X\times B$. As the coefficients of $P_y$ are regular functions in the variable $y$, fixing $x\in A$ we conclude that $P_y(x,f(x,y))=0$ holds for all $(x,y)\in A \times Y$. Now fixing $y$ we see that it holds for all $(x,y)\in X\times Y$. This shows that
\begin{equation*}
    Q(x,y,t):=P_y(x,t)
\end{equation*}
is a non-trivial regular function of three variables $(x,y,t)\in X\times Y\times \K$ which vanishes on the graph of $f$ and is of degree at most one with respect to $t$. This gives the desired rational representation of $f$.
\end{proof}

\subsection{Proof of Theorem \ref{thm_Nash}}
Throughout this section let $X,Y$ be two semialgebraic open subsets of $R^{n_X}$ and $R^{n_Y}$ respectively for some $n_X,n_Y\geq 0$, with $X$ semialgebraically connected, as in Theorem \ref{thm_Nash}. First we need to repeat the proof of Theorem \ref{thm_reg} to arrive with a similar result in the Nash case.

\begin{lem}\label{graph_is_salgebraic}
Let $f:X\times Y\rightarrow R$ satisfy the assumptions of Theorem \ref{thm_Nash}. Then the graph of $f$ is contained in a semialgebraic subset of $X\times Y\times R$ of codimension one.
\begin{proof}
By the semialgebraic cell decomposition (\cite[Proposition 9.1.8]{RAG}) $Y$ can be written as a finite disjoint union of semialgebraic subsets, each Nash diffeomorphic to $R^n$ for some $n\leq n_Y$. Hence without loss of generality we might assume that $Y=R^{n_Y}$.

Let $S:=X\times R$ and let $V$ be the vector space of all functions on $X\times R$ which are restrictions of polynomials from $R^{n_X+1}$. Fix some basis $(e_i)_{i\in\mathbb{N}}$ of $V$. For $Q\in V\backslash\{0\}$ define $c(Q)$ as in Lemma \ref{main_lemma} and for a subset $F\subset X\times K$ define $d(F)$ also as in Lemma \ref{main_lemma}.

As for every $y\in Y$ the function $f(\cdot,y)$ is Nash we have that $c(f(\cdot,y))<\infty$, so by Example \ref{alg_second_category} there exists a natural number $n$ such that the set
\begin{equation*}
    B:=\{y\in Y:c(\Gamma_{f(\cdot,y)})=n\}
\end{equation*}
is Zariski dense in $Y$. Fix $y_0\in B$. Applying Lemma \ref{main_lemma} with $F=\Gamma_{f(\cdot,y_0)}$, we find polynomial functions $(P_i)_{i=1,\dots,n}\in \mathcal{P}(R^{(n_X+1)(n-1)})$ such that for all points $x_1,\dots,x_{n-1}\in X$ the function
\begin{equation*}
    P_{y_0}:=\sum_{i=1}^n P_i(x_1,f(x_1,y_0),\dots,x_{n-1},f(x_{n-1},y_0))e_i\in V
\end{equation*}
satisfies $P_{y_0}(x,f(x,y_0))=0$ for $x\in X$. Now the set of tuples $(x_1,\dots,x_{n-1})\in X^{n-1}$ for which $P_{y_0}$ is the trivial polynomial is the zero set of some Nash function on $X^{n-1}$ and by assumption is non-empty. This means that we can find a tuple which belongs to $A^{n-1}$ for which $P_{y_0}$ is non-trivial. Fix such a choice of $(x_1,\dots,x_{n-1})$. 

The equation $P_y(x,f(x,y))=0$ by construction holds for every $(x,y)\in X \times B$. As the coefficients of $P_y$ are Nash functions in the variable $y$, fixing $x\in A$ we conclude that $P_y(x,f(x,y))=0$ holds for all $(x,y)\in A\times Y$. Now fixing $y$ we see that it holds for all $(x,y)\in X\times Y$. This shows that
\begin{equation*}
    Q(x,y,t):=P_y(x,t)
\end{equation*}
is a non-trivial Nash function of three variables $(x,y,t)\in X\times Y\times R$ which vanishes on the graph of $f$. As $X\times Y\times R$ is semialgebraically connected, the zero set of $Q$ is the desired set of codimension one containing $\Gamma_f$.
\end{proof}
\end{lem}

Now we will introduce a few auxiliary results which will allow us to show that Theorem \ref{thm_Nash} follows from Lemma \ref{graph_is_salgebraic}.

\begin{obs}\label{cont_ext}
Let $U\subset X\times Y$ be a semialgebraic set such that $U\cap (X\times \{y\})$ is dense in $X\times \{y\}$ for every $y\in Y$. Let $g:U\rightarrow R$ be a semialgebraic function. For $y_0\in Y$ we define $g(\cdot,y_0)$ as the function $x\mapsto g(x,y_0)$ defined on the set $\{x\in X: (x,y_0)\in U\}$. Then the set
\begin{equation*}
    Y_0:=\{y\in Y: g(\cdot,y)\text{ can be extended to a continuous function on }X\}
\end{equation*}
is semialgebraic and moreover the unique extension of $g\vert_{U\cap (X\times Y_0)}$ to $X\times Y_0$ which is continuous with respect to the first variable is semialgebraic. 
\begin{proof}
Consider the embedding $R\subset \mathbb{P}^1(R)=R\cup\{\infty\}$ in the projective space. As the operation of taking the closure of a set is definable, the following set is semialgebraic:
\begin{equation*}
    G:=\{(x,y,t)\in X\times Y\times \mathbb{P}^1(R):(x,t)\in \overline{\Gamma_{f(\cdot,y)}}\}.
\end{equation*}
Now $Y_0$ can be written as
\begin{equation*}
    Y_0=\{y\in Y: \forall_{x\in X} (x,y,\infty)\not\in G \wedge \exists !_{t\in R} \,(x,y,t)\in G\}
\end{equation*}
so it is semialgebraic. Finally the graph of the extension of $g\vert_{U\cap (X\times Y_0)}$ is precisely $G\cap(X\times Y_0\times R)$ so the function is semialgebraic.
\end{proof}
\end{obs}

\begin{prop}\label{cont_is Nash}
Let $g:X\times Y\rightarrow R$ be a semialgebraic function, continuous with respect to the variable $x\in X$. Then the set of points $y\in Y$ for which $g(\cdot,y)$ is Nash is a semialgebraic subset of $Y$.
\begin{proof}
According to \cite[Theorem 8.10.5]{RAG} for a fixed $g$ there exists a natural number $r$ such that for $y\in Y$, $g(\cdot,y)$ is Nash if and only if it is of class $\mathcal{C}^r$. This condition can easily be verified by a first order formula as it depends only on finitely many partial derivatives of $g(\cdot,y)$. 
\end{proof}
\end{prop}

\begin{proof}[Proof of Theorem \ref{thm_Nash}]
Let $f:X\times Y\rightarrow R$ satisfy the assumptions of Theorem \ref{thm_Nash}. We will use induction on the dimension of $Y$, the case $n_Y=0$ being trivial. Let $Z\subset Y$ be a semialgebraic closed subset of $Y$ of codimension at least one. Thanks to the cell decomposition (\cite[Proposition 9.1.8]{RAG}), $Z$ can be written as a disjoint union of its semialgebraic subsets, each of them Nash diffeomorphic to $R^n$ for some $n<n_Y$. Now applying the induction hypothesis to each of the cells separately we get that $f\vert_{X\times Z}$ is semialgebraic. This means that every time in the proof we encounter a semialgebraic subset of $Y$ of codimension at least one, relaying on the induction hypothesis we can delete its closure $Z$ from $Y$ and consider $Y'=Y\backslash Z$ instead.

Let $F\subset X\times Y\times R$ be a semialgebraic set of codimension one containing the graph of $f$, which exists by Lemma \ref{graph_is_salgebraic}. By the cell decomposition (\cite[Proposition 9.1.12]{RAG}), we can find a finite family $(U_i)_{i=1,\dots,n}$ of semialgebraic open disjoint sets covering a dense subset of $X\times Y$ that for each $1\leq i \leq n$ there exists a finite number of semialgebraic continuous functions $(\xi_{i,j})_{j=1,\dots,n_i}:U_i\rightarrow R$ such that for $(x,y)\in U_i$ we have
\begin{gather*}
    1\leq j_1<j_2\leq n_i \implies \xi_{i,j_1}(x)<\xi_{i,j_2}(x) \\
    \text{for } t\in R, (x,y,t)\in F \iff t=\xi_{i,j}(x,y)\text{ for some }j.
\end{gather*}
Subdividing the family further, by \cite[Lemma 8.10.6]{RAG} we can assume that for each $y\in Y$ and $1\leq i \leq n$ we have that the set $U_i\cap (X\times \{y\})$ is either empty or semialgebraically connected. Define $U:=\bigcup U_i$ and $C:=X\times Y \backslash U$.

Consider
\begin{equation*}
    \{y\in Y: \dim (C\cap (X\times \{y\}))=n_X\}.
\end{equation*}
It is a semialgebraic subset of $Y$ of codimension at least one. Relying on the induction hypothesis, as stated at the beginning of the proof, we might delete its closure from $Y$ to assume that
\begin{equation*}
    \text{ for every }y\in Y,\; U\cap (X\times \{y\})\text{ is dense in }X\times\{y\}.
\end{equation*}
Similarly now consider 
\begin{equation*}
    \{x\in X: \dim(C\cap (\{x\}\times Y))=n_Y\}.
\end{equation*}
It is a semialgebraic subset of $X$ of codimension at least one. This means that we can find some $x_0\in A$ such that
\begin{equation*}
    \dim(C\cap (\{x_0\}\times Y))<n_Y.
\end{equation*}
Again relying on the induction hypothesis, we can delete the set $\{y\in Y:(x_0,y)\in C\}$ from $Y$ to assume that
\begin{equation}\label{x_0_capC=empty}
    C\cap (\{x_0\}\times Y)=\emptyset.
\end{equation}

Now for every collection of natural numbers $K=(k_i)_{i=1,\dots,n}$ satisfying $1\leq k_i\leq n_i$ for every $i$ define the semialgebraic function $ g_K:U\rightarrow R$ by
\begin{equation*}
    g_K(x,y):=\begin{cases}
        \xi_{1,k_1}(x,y) & \text{ for }(x,y)\in U_1 \\
        \dots \\
        \xi_{n,k_n}(x,y) & \text{ for }(x,y)\in U_n.
    \end{cases}
\end{equation*}
Consider
\begin{equation*}
    B_K:=\{y\in Y: g_K(\cdot,y)\text{ can be extended to a continuous function on X}\}.
\end{equation*}
By Observation \ref{cont_ext} the set is semialgebraic. Let $\tilde{g}_K$ denote the extension of $g_K$ to $X\times B_k$ which is continuous with respect to the variable $x$. Again by Observation \ref{cont_ext} the function is semialgebraic. Finally consider:
\begin{equation*}
    \widetilde{B}_K:=\{y\in Y: \tilde{g}_K(\cdot,y)\text{ is Nash and } \tilde{g}_K(x_0,y)=f(x_0,y)\}
\end{equation*}
By Proposition \ref{cont_is Nash} the set is semialgebraic.

Now fix $y_0\in Y$. As the function $f(\cdot,y_0)$ is continuous and the sets $U_i\cap (X \times \{y_0\})$ are semialgebraically connected, for every $i$ we must have that there exists $1\leq j \leq n_i$ such that $f(x,y_0)=\xi_{i,j}(x,y_0)$ holds for $x$ with $(x,y_0)\in U_i$. This implies that there exists a multiindex $K$ such that $y_0\in \widetilde{B}_k$ and $\tilde{g}_K(x,y_0)=f(x,y_0)$ holds for $x\in X$.

This in particular shows that the sets $\widetilde{B}_K$ cover the entire $Y$. Note that if $y\in \widetilde{B}_K\cap \widetilde{B}_{K'}$ for some two multiindices $K$ and $K'$ then 
\begin{equation*}
    \tilde{g}_K(x,y)=\tilde{g}_{K'}(x,y)\text{ for }x\in X
\end{equation*}
as both of the functions are Nash with respect to $x$ and agree on a neighbourhood of $x_0$ (thanks to (\ref{x_0_capC=empty})). This shows that $f$ can be written as
\begin{equation*}
    f(x,y)=
    \begin{cases}
        \tilde{g}_{K_1}(x,y) & \text{ for }y\in \tilde{B}_{K_1}\\
        \dots \\
        \tilde{g}_{K_N}(x,y) & \text{ for }y\in \tilde{B}_{K_N}
    \end{cases}
\end{equation*}
where $(K_i)_{i=1,\dots,N}$ ranges through all the possible multiindices. This proves that $f$ is semialgebraic. 
\end{proof}
\section{Acknowledgements}
The author was partially supported by the National Science Centre (Poland) under grant number 2022/47/B/ST1/00211.


\normalsize


\begin{thebibliography}{[HD82]}



\normalsize
\baselineskip=17pt



\bibitem[1]{RAG} Jacek Bochnak, Michel Coste, and Marie-Fran\c{c}oise Roy. \emph{Real Algebraic Geometry.} Springer Berlin Heidelberg, Berlin, Heidelberg, 1998.
\bibitem[2] {Gwozdziewicz} Beata Gryszka and Janusz Gwo\'zdziewicz. \emph{On some regularity condition}.
Period. Math. Hungar., 86:336–342, 2023.
\bibitem[3]{curve} J\'anos Koll\'ar, Wojciech Kucharz, and Krzysztof Kurdyka. \emph{Curve-rational
functions.} Mathematische Annalen, 370:39–69, Feb 2018.
\bibitem[4]{separatelyNash} Wojciech Kucharz, Krzysztof Kurdyka, and Ali El-Siblani. \emph{Separately nash
and arc-nash functions over real closed fields.} Bulletin of the London Mathematical Society, 53:426–441, 2021.
\bibitem[5]{Palais} Richard S. Palais. \emph{Some analogues of Hartogs’ Theorem in an algebraic
setting.} American Journal of Mathematics, 100:387–405, 1978.
\end{thebibliography}
\end{document}